\documentclass[10pt, a4paper]{amsart}

\usepackage[DIV10, headinclude, BCOR 10mm]{typearea}
\usepackage[colorlinks,linkcolor=blue,citecolor=blue,urlcolor=blue, pdfcenterwindow, pdfstartview={XYZ null null 1.2}, pdffitwindow, pdfdisplaydoctitle=true]{hyperref}
\usepackage[english]{babel}
\usepackage{array}
\usepackage{threeparttable}
\usepackage{rotating}
\usepackage[section]{placeins}
\usepackage{amssymb} 
\usepackage{mathrsfs}
\usepackage{enumerate}
\usepackage{color}
\usepackage{relsize}
\usepackage{amsrefs}        
\usepackage[all]{xy}

\SelectTips{cm}{}
\hypersetup{pdfauthor={Alexander Bendikov, Christophe Pittet, Roman Sauer},pdftitle={Spectral distribution to L2-isoperimetric profile}}

\newcommand{\calh}{{\mathcal H}}

\newcommand{\caln}{{\mathcal N}}

\newcommand{\bbC}{{\mathbb C}}
\newcommand{\bbR}{{\mathbb R}} 
\newcommand{\bbZ}{{\mathbb Z}} 
\newcommand{\bbN}{{\mathbb N}}

\newcommand{\abs}[1]{{\left\lvert #1\right\rvert}}

\newcommand{\id}{\operatorname{id}}

\renewcommand{\Im}{\operatorname{im}}

\newcommand{\pr}{\operatorname{pr}}

\newcommand{\norm}[1]{{\left\lVert #1\right\rVert}}
\newcommand{\fixnorm}[1]{{\lVert #1\rVert}}
\newcommand{\supp}{\operatorname{supp}}
\newcommand{\trace}{\operatorname{tr}}

\DeclareMathOperator{\folner}{F{\o}l}

\newtheorem{theorem}{Theorem}[section]
\newtheorem{lemma}[theorem]{Lemma}
\newtheorem{proposition}[theorem]{Proposition}
\newtheorem{corollary}[theorem]{Corollary}
\newtheorem{definition}[theorem]{Definition}

\newtheorem{remark}[theorem]{Remark}

\numberwithin{equation}{section}

\begin{document}
\title[Spectral distribution and L$^\mathrm{2}$-isoperimetric profile]{Spectral distribution and L$^\mathrm{2}$-isoperimetric profile of Laplace operators on groups}

\author{Alexander Bendikov}
\author{Christophe Pittet}
\author{Roman Sauer}

\address{Institute of Mathematics, Wroclaw University}
\email{bendikov@math.uni.wroc.pl}

\address{CMI Universit\'e d'Aix-Marseille I, Marseille}
\curraddr{Laboratoire Poncelet CNRS, Moscow}
\email{pittet@cmi.univ-mrs.fr}

\address{Department of Mathematics, University of Chicago}    
\curraddr{Mathematisches Institut, WWU M\"unster}
\email{sauerr@uni-muenster.de}

\dedicatory{To the memory of Andrzej Hulanicki}

\keywords{$L^2$-isoperimetric profile, Random walk, Spectral distribution of Laplace-operators}

\subjclass[2000]{Primary: 58J35, 20F63; Secondary: 60G50}

\thanks{A.~Bendikov was supported by the University of Aix-Marseille I
as an invited Professor and by the Polish
Government Scientific Research Fund, Grant NN201371736. 
Ch.~Pittet was supported by the CNRS and the Marie Curie Transfer of Knowledge Fellowship of the European Community's Sixth Framework
Program under contract number MTKD-CT-2004-013389
with the University of Wroclaw.
A.~Bendikov and Ch.~Pittet are grateful to Prof. E.~Damek who managed the ToK contract, and to Prof. A. Grigor'yan for an invitation at the University of Bielefeld. They are also grateful to the Erwin Schr\"odinger Institute for several
invitations. R.~Sauer was supported by DFG grant SA 1661/1-2. All authors are grateful for the financial support 
from Prof.~W.~L\"uck's Leibniz award for a meeting at the WWU M\"unster.}

\begin{abstract}
We give a formula relating the $L^2$-isoperi\-metric profile to the spectral distribution of the 
Laplace operator associated to a finitely generated group $\Gamma$ 
or a Riemannian manifold with a cocompact, isometric $\Gamma$-action. 
As a consequence, we can apply techniques from geometric group theory
to estimate the spectral distribution of the Laplace operator 
in terms of the growth and the F{\o}lner's function 
of the group, 
generalizing previous estimates by Gromov and Shubin.  
This leads, in particular, to sharp estimates of  
the spectral distributions for several classes of solvable groups. Furthermore, 
we prove the asymptotic invariance of the spectral distribution under changes 
of measures with finite second moment. 
\end{abstract}

\maketitle

\section{Introduction}% (fold)
\label{sec:intro}

What is the relation between the asymptotic behavior of the \emph{return probability} $p(t)$ of the random walk of a probability measure $\mu$ on a finitely generated group $\Gamma$, the \emph{$L^2$-isoperimetric profile} $\Lambda(v)$ of the Laplace operator $\Delta$ associated to $\mu$, 
and the \emph{spectral distribution} $N(\lambda)$ of $\Delta$? \smallskip\\
\indent Recall that one has the equality 
\begin{equation}\label{eq:easy relation}
 p(2t)=\mu^{(2t)}(\{e\})=\int_0^\infty (1-\lambda)^{2t}dN(\lambda),~~t\in\bbN,
\end{equation}
between $p(2t)$, the measure of the unit element $e\in\Gamma$ with respect to the 
$2t$-th convolution power of $\mu$, and $N(\lambda)$ (see also~\eqref{eq:measure and density} 
below). However, deducing from the above equality a relation between the asymptotic behaviors of 
$p(2t)$ and $N(\lambda)$ is in general a difficult task involving Tauberian
theory (see~\citelist{\cite{BGT}\cite{GroShu}*{Appendix}}). The relationship between $p(2t)$ and $\Lambda(v)$ 
has been essentially settled in a series of works by Coulhon and Grigor'yan~\cite{Cou-survey}. 
They prove under a mild regularity assumption that the function $\gamma(t)$ defined by the equation
\begin{equation}\label{eq:Coulhon-Grigoryan relation}
	t=\int_1^{\gamma(t)}\frac{dv}{\Lambda(v)v}~, t\ge 0,  
\end{equation}
satisfies $\frac{1}{\gamma(t)}\simeq p(2t)$ for $t\in\bbN$ near infinity in the sense of Section~\ref{sub:basic notions}. 

The asymptotic relation between $p,\Lambda$, and $N$ is fully understood 
for non-amenable and for virtually nilpotent groups due to work of 
Kesten, Varopoulos and Gromov-Shubin, respectively. The computation of 
$p(2t)$ and $\Lambda(v)$ is a field of active research 
(see~\cites{pittet-jems, gromov-preprint, pittet-coulhon, erschler,erschler-piecewise, CouGri} -- just to name 
a few). 

We combine the asymptotic relations between $N(\lambda)$ and $p(2t)$,  
obtained from~\eqref{eq:easy relation} via Legendre transform techniques,  
and between $p(2t)$ and $\Lambda(v)$ 
in~\eqref{eq:Coulhon-Grigoryan relation} 
to 
establish a surprisingly 
simple formula relating $N(\lambda)$ and $\Lambda(v)$ that holds under a mild 
regularity assumption. This formula (Theorem~\ref{thm:isospectral}), which is our main contribution, 
leads to explicit estimates of $N(\lambda)$ for many 
examples (see Table~\ref{table:computations}). Furthermore, we show that the 
asymptotic behavior of $N(\lambda)$ is stable under changes of the measure $\mu$ as long as 
$\mu$ is symmetric, has finite second moment, and its support generates 
$\Gamma$ (Theorem~\ref{thm:second moment}).  

\subsection{Basic notions}\label{sub:basic notions} Some definitions are in order to state the precise 
result. Let $\Gamma$ be a finitely generated group. A probability measure $\mu$ 
on $\Gamma$ is called $\emph{symmetric}$ if $\mu(\{\gamma\})=\mu(\{\gamma^{-1}\})$ 
for every $\gamma\in\Gamma$. It is said to have \emph{finite second moment} if 
$\int_\Gamma l(\gamma)^2d\mu(\gamma)<\infty$, where $l$ denotes the length function associated to some word metric on $\Gamma$. 
Further, we say that $\mu$ is \emph{admissible} if it is symmetric, 
has finite second moment, and its support contains a finite generating 
set. 
Let $l^2(\Gamma)$ be the Hilbert space of square-integrable, complex-valued  functions on $\Gamma$. 
Right convolution with $\mu$ defines a self-adjoint operator (\emph{Markov-operator})
\[
	R_\mu:l^2(\Gamma)\rightarrow l^2(\Gamma), ~~R_\mu(f)(x)=\sum_{\gamma\in\Gamma}f(x\gamma^{-1})\mu(\gamma). 
\]
with operator norm bounded by $1$. 
The \emph{Laplace operator} $\Delta$ of $\mu$ is 
the positive operator defined as $\Delta=\id-R_\mu$. Both $\Delta$ and 
$R_\mu$ lie in the \emph{von Neumann algebra} $\caln(\Gamma)$ of $\Gamma$ 
which is defined as the algebra of bounded operators on $l^2(\Gamma)$ that 
are equivariant with respect to the obvious isometric left $\Gamma$-action 
on $l^2(\Gamma)$. The \emph{von Neumann trace} $\trace_\Gamma:\caln(\Gamma)\rightarrow\bbC$ is defined as 
$\trace_\Gamma(A)=\langle A(\delta_e),\delta_e\rangle_{l^2(\Gamma)}$. 
For a self-adjoint operator $A\in\caln(\Gamma)$ the spectral projection 
$E^A_\lambda=\chi_{(-\infty,\lambda]}(A)$ of $A$ (see \emph{e.g.}~\cite{lueck(2002a)}*{p.~56}) lies in $\caln(\Gamma)$. 
The \emph{spectral distribution} of $\Delta$ is the right-continuous 
function $N:[0,\infty)\rightarrow [0,\infty)$ with 
\[
	N(\lambda)=\trace_\Gamma\bigl(E^\Delta_\lambda\bigr). 
\] 
The probability to return to $e\in\Gamma$ after $t$ steps 
of the random walk defined by $\mu$ and 
starting at $e\in\Gamma$ is called the \emph{return probability} $p(t)$ of $\mu$. 
The return probability can be expressed as 
\begin{align}\label{eq:measure and density}
	p(2t)=\mu^{(2t)}(\{e\})=\langle R_\mu^{2t}(\delta_e),\delta_e\rangle_{l^2(\Gamma)}
	= \trace_\Gamma\bigl(R_\mu^{2t}\bigr)&= \trace_\Gamma\bigl(\int_0^\infty (1-\lambda)^{2t}dE_\lambda^\Delta\bigr)\notag\\
	&= \int_0^\infty (1-\lambda)^{2t}dN(\lambda). 
\end{align}
The \emph{$L^2$-isoperimetric profile} of $\Delta$ is the function 
$\Lambda:[1,\infty)\rightarrow (0,\infty)$ such that $\Lambda(v)$ is, 
by definition, 
the smallest eigenvalue of $\Delta$ restricted to 
a set $\Omega\subset\Gamma$ of cardinality less or equal to $v$: 
\[\Lambda(v)=\inf_{1\leq|\Omega|\leq v}\lambda_1(\Omega)
~\text{ where }~
\lambda_1(\Omega)=\inf_{\emptyset\neq\supp(f)\subset\Omega}\frac{\langle\Delta(f),f\rangle}{\|f\|_2^2}.\]
If the choice of the measure $\mu$ in the definition of $N$, $\Lambda$, $\Delta$, or $p$ needs to be specified, we write $N_\mu$, $\Lambda_\mu$, $\Delta_\mu$, and $p_\mu$. 

The function $\Lambda$ is a decreasing right-continuous step function.
It is stable under changes of measures; more precisely, if $\mu$ and $\nu$ are 
two admissible measures on $\Gamma$, then there exists constants $C\geq 1$ and $c>0$,
such that for all $v\geq 1$,
\begin{equation}\label{equ:stability of the isospectral profile}
c\Lambda_{\mu}(v)\leq\Lambda_{\nu}(v)\leq C\Lambda_{\mu}(v).	
\end{equation}
This follows from \cite[Proposition 4]{Gretete}.
For finitely generated groups, 
Cheeger's inequality (Theorem~\ref{thm:Cheeger inequality})
and F{\o}lner's characterization of amenability imply that
$\Lambda(v)\rightarrow 0$ for $v\rightarrow\infty$ 
if and only if the group is amenable.\smallskip\\
We say that $f\preceq g$ holds \emph{near zero},  
for functions $f,g:[0,\infty)\rightarrow [0,\infty)$, 
if there are constants $C,D,\epsilon>0$ 
such that $f(\lambda)\le Cg(D\lambda)$ for all $\lambda\in [0,\epsilon)$. 
We say that $f\preceq g$ holds \emph{near infinity} 
if there are constants $C,D,x_0>0$ 
such that $f(x)\le Cg(Dx)$ for all $x\in [x_0,\infty)$.
Further we write $f\simeq g$ \emph{near zero} or \emph{near infinity} 
if $f\preceq g$ and $g\preceq f$ 
hold near zero or near infinity, respectively. 
We say that \emph{$f\simeq g$ holds in the dilatational sense} 
or \emph{$f\simeq g$ are dilatationally equivalent} if the outer constant 
$C$ can be taken to be $C=1$ in the above definition. Similarly, 
one defines $\preceq$ in the dilatational sense. The same definitions apply 
for functions $f,g:\bbN\to [0,\infty]$ by considering their piecewise linear 
extensions. 

\begin{remark}
The return probability and the $L^2$-isoperimetric profile can be defined for 
measures on arbitrary graphs (not just Cayley graphs) whilst the definition of $N(\lambda)$ 
uses in an essential way the trace on the von Neumann algebra of the group. 
Similarly, the heat kernel and the $L^2$-isoperimetric profile are defined for complete 
Riemannian manifolds. Actually, their relationship was first extensively studied in this context 
by Grigor'yan~\cite{Gri-revista}. The definition of 
the spectral distribution $N(\lambda)$ on 
complete Riemannian manifolds requires the existence of 
a proper, free, cocompact, isometric group action~\cite{GroShu}. 
\end{remark}

\subsection{The formula} 

The computations of the 
spectral distribution and the $L^2$-isoperimetric profile
of virtually nilpotent groups are classic results (see the first 
row of Table~\ref{table:computations}) due to the work of 
Gromov-Shubin and Varopoulos~\cites{GroShu,varopoulos-book}.  

\begin{theorem}[Gromov-Shubin-Varopoulos -- reformulated]\label{thm:Gromov Shubin reformulated}
Let $\Gamma$ be an infinite finitely generated group, and 
let $\mu$ be an admissible probability measure on $\Gamma$. 
Assume that there is $\alpha\in (-\infty,0)$ such that 
$\Lambda_\mu(v)\leq v^\alpha$ near infinity. Then $\Gamma$ is virtually nilpotent 
and, near zero,    
\[N_\mu(\lambda)\simeq\frac{1}{\Lambda_\mu^{-1}(\lambda)}.\] 
In fact, we have $\Lambda_\mu(v)\simeq v^{-2/d}$ near infinity 
and $N_\mu(\lambda)\simeq \lambda^{d/2}$ near zero, 
where $d$ is the degree of growth of $\Gamma$. 
\end{theorem}
\begin{proof} According to the inequality~\eqref{equ:stability of the isospectral profile} above, we may assume that the probability measure $\mu$ is uniform with support a finite symmetric generating set $S$ of $\Gamma$. Let $B(r)\subset\Gamma$ denote the ball of radius $r$ around the identity element of $\Gamma$ with respect to the word metric defined by $S$. 
	Let $\Phi$ be the inverse of the growth function, that is for 
$v\geq 0$,
\[\Phi(v)=\min\{r:~ |B(r)|>v\}.
\]
We have for large enough $v>1$: 
\begin{align}
v^{\alpha}\geq \Lambda(v)&=\inf_{|\Omega|\leq v}\lambda_1(\Omega)\notag\\
&\geq\inf_{|\Omega|\leq v}\,\inf_{\omega\subset\Omega}\frac{1}{2|S|^2}
\left(\frac{|\partial_S\omega|}{|\omega|}\right)^2=\inf_{|\omega|\leq v}\frac{1}{2|S|^2}
\left(\frac{|\partial_S\omega|}{|\omega|}\right)^2\label{eq:use of Cheeger}\\
&\geq\inf_{|\omega|\leq v}\frac{1}{2|S|^2}
\left(\frac{1}{4|S|\Phi(2|\omega|)}\right)^2\geq\frac{1}{2|S|^2}
\left(\frac{1}{4|S|\Phi(2v)}\right)^2;\label{eq:coulhon-saloffcoste}
\end{align}
the inequality~\eqref{eq:use of Cheeger} is Cheeger's inequality (Theorem~\ref{thm:Cheeger inequality}); the first inequality in~\eqref{eq:coulhon-saloffcoste}
is due to Coulhon and Saloff-Coste~\cite{CouSal} (see also \cite[Theorem 3.2]{pittet-isop}). The above chain of inequalities implies that the growth rate of the balls
in $\Gamma$ is bounded above, up to a multiplicative 
constant, by $v^{-2/{\alpha}}$. Gromov's
theorem on polynomial growth~\cite{GroGro}, implies that $\Gamma$ is virtually nilpotent.  The lower 
bound  $\Lambda_\mu(v)\geq v^{-2/d}$ near infinity,
follows from the above chain of inequalities because $\Phi(v)\simeq v^{1/d}$ near 
infinity \cite[Appendix, Proposition 2]{GroGro}.
The upper bound $\Lambda_\mu(v)\leq v^{-2/d}$ is proved in \cite[2.12]{pittet-coulhon}.
Starting from Varopoulos estimate $p(2t)\simeq t^{-d/2}$ near infinity, Gromov and Shubin \cite{GroShu} apply Karamata
type methods to deduce that $N_\mu(\lambda)\simeq \lambda^{d/2}$ near zero (see also 
\cite[Lemma 2.46]{lueck(2002a)}).   
\end{proof}

The reason 
we interpret Gromov-Shubin's computation of $N(\lambda)$ 
via the inverse of the $L^2$-isoperimetric profile in the 
above theorem is that it is, as turns out through the present work, 
suited for generalization. 

The following theorem is our main result (proved 
in Section~\ref{sec:isospectral}). Notice 
that a function like $\Lambda(v)=\log(v)^{-\gamma}$ with $\gamma>0$ does not 
satisfy the assumption in Theorem~\ref{thm:Gromov Shubin reformulated} but the one 
of Theorem~\ref{thm:isospectral} (see Table~\ref{table:computations} for many more such 
examples). 
This opens the way for many computations beyond nilpotent groups. 

\begin{theorem}\label{thm:isospectral}
Let $\Gamma$ be an infinite finitely generated amenable group, and 
let $\mu$ be an admissible probability measure on $\Gamma$. 
Assume that the function $\Lambda_\mu\circ\exp$ is doubling  
near infinity (Definition~\ref{def:generalized inverse}). 
Then we have the following dilatational equivalence near zero between 
$N_\mu$ and the reciprocal of the generalized inverse (Definition~\ref{def:generalized inverse}) of $\Lambda_\mu$: 
\[N_\mu(\lambda)\simeq\frac{1}{\Lambda_\mu^{-1}(\lambda)}.\] 
\end{theorem}

Actually, Theorem~\ref{thm:isospectral} follows from the following more 
general statement: 

\begin{theorem}\label{thm:isospectral general}
Retain the setting of Theorem~\ref{thm:isospectral}. 
Let $L:(0,\infty)\rightarrow (0,\infty)$ be a decreasing function such 
that $\lim_{x\rightarrow\infty}L(x)=0$. 
Assume that $L\circ\exp$ is doubling near infinity. 
\begin{enumerate}[(1)]
	\item If $L(v)\preceq\Lambda_{\mu}(v)$ near infinity, then $N_\mu(\lambda)\preceq\frac{1}{L^{-1}(\lambda)}$ near zero in the dilatational sense. 
	\item If $L(v)\succeq\Lambda_{\mu}(v)$ near infinity, then $N_\mu(\lambda)\succeq\frac{1}{L^{-1}(\lambda)}$ near zero in the 
	dilatational sense. 
\end{enumerate} 
\end{theorem}

So far we do not know an example of amenable $\Gamma$ and $\mu$ such that 
$\Lambda_\mu$ neither satisfies the hypothesis 
of Theorem~\ref{thm:Gromov Shubin reformulated} 
nor the one of Theorem~\ref{thm:isospectral} or~\ref{thm:isospectral general}. 

There is an analogous version of Theorem~\ref{thm:isospectral} (similarly, of 
Theorem~\ref{thm:isospectral general}) in the Riemannian setting: 
 
\begin{corollary}\label{cor:isospectral for manifolds}
	Let $M$ be a connected 
	complete non-compact Riemannian manifold and $\Gamma$ be an 
	amenable group. 
	Let $\Gamma$ act freely, properly discontinuously 
	and with cocompact quotient on $M$ by isometries. 
	Let $\Lambda$ be the $L^2$-isoperimetric 
	profile of $M$ and $N(\lambda)$ be the spectral distribution of the Laplace 
	operator of $M$ on functions. Let $\Lambda^{-1}$ be the generalized 
	inverse of $\Lambda$. If $\Lambda\circ\exp$ is doubling near infinity, then,  near zero, 
	\[
		N(\lambda)\simeq\frac{1}{\Lambda^{-1}(\lambda)}. 
	\]	
\end{corollary} 

Indeed, Efremov proved that the spectral distributions of the Riemannian 
Laplace operator on $M$ and the combinatorial Laplace operator $\Delta_\mu$ 
on $\Gamma$ for 
the probability measure 
\begin{equation}\label{eq:simple random walk}
\mu=\frac{1}{\abs{S}}\sum_{s\in S}\delta_s
\end{equation}
of the simple random walk associated to a finite, symmetric, generating set 
$S$ are 
equivalent near zero~\cite{lueck(2002a)}*{Section~2.4}. The corresponding 
statement for the $L^2$-isoperimetric profile can be deduced 
from~\cites{CouGri,revista}.

\subsection{Exponential and sub-exponential growth, F{\o}lner's function, and almost flat spectra} 
If not specified otherwise below, we consider the probability 
measure~\eqref{eq:simple random walk} associated to a finite, symmetric 
generating set of the group in consideration. 
According to Theorem~\ref{thm:second moment} we may as well 
take any other admissible probability measure as long as 
we are only interested 
in the asymptotic properties of $N_\mu$.  

\smallskip

Geometric methods allow to compute $\Lambda(v)$ and verify the doubling 
assumption in many cases -- the computation usually uses information 
about \emph{F{\o}lner's function} 
$\folner:(0,\infty)\rightarrow\bbN$ defined for finitely generated amenable groups by 
\begin{equation}\label{eq:Folner function}
	\folner(r)=\min\Bigl\{\abs{\Omega};~\Omega\subset\Gamma: \frac{|\partial_S\Omega|}{|\Omega|}<\frac{1}{r}\Bigr\}.
\end{equation}
Here the \emph{boundary} $\partial_S\Omega$ of $\Omega$ is, by definition, 
$\partial_{S}\Omega=\{x\in\Omega;~\exists s\in S: xs\in\Gamma\setminus\Omega\}$.
The function 
$\folner(r)$ is an increasing, right-continuous step function. It satisfies 
$\folner(r)\rightarrow\infty$ as $r\rightarrow\infty$, if and only if the group $\Gamma$ is infinite. If 
the F{\o}lner's function grows sufficiently fast, we can deduce an upper 
bound on $N(\lambda)$: 

\begin{proposition}\label{prop:folner decay implies doubling condition}
Let $\Gamma$ be an infinite, finitely generated amenable group. 
Let $F:(x_0, \infty)\rightarrow (0,\infty)$ be 
a continuous, strictly increasing function such that 
\begin{enumerate}[(1)]
\item $F(r)\rightarrow\infty$ as $r\rightarrow\infty$,  
\item there is $C>1$ such that 
$F(r)^2\le F(Cr)$ for large $r>0$, 
\item $F(r)\preceq\folner(r)$ near infinity. 
\end{enumerate}
Then, near zero, 
\[
	N(\lambda)\preceq \frac{1}{F\bigl (\lambda^{-1/2}\bigr)}
\]
in the dilatational sense. 
\end{proposition} 

\begin{remark}
	In all known examples there is the equivalence 
	\[
		\folner(r)\simeq \Lambda^{-1}\bigl(\frac{1}{r^2}\bigr)  
	\]
	near infinity, hence condition (2) can be considered dual to the doubling condition on 
	$\Lambda\circ\exp$ in Theorem~\ref{thm:isospectral general}, which itself is 
	equivalent to the existence of a constant $c>0$ such that 
	$\Lambda(v^2)\ge c \Lambda(v)$ for $v\ge 1$ large enough. 
\end{remark}

The proof of the above proposition is based on 
Theorem~\ref{thm:isospectral general} and Cheeger's inequality 
(see Section~\ref{sec:isoperimetric profile and geometry}). 
We refer the reader to~\cites{CouSal,erschler,gromov-preprint}  
for lower bounds on F{\o}lner's function.
For lower bounds on $N(\lambda)$ we refer to Theorem~\ref{thm:isospectral general} and 
Proposition~\ref{pro:lower bound on N} in 
Section~\ref{sec:isoperimetric profile and geometry}. 

\begin{corollary}\label{cor:Coulhon Saloff-Coste inequality}
Let $\Gamma$ be a finitely generated amenable group. 
Let $0<\alpha\leq 1$. 
Assume there exists $\epsilon>0$ such that the cardinality of a ball of radius $r$ 
in $\Gamma$ is bounded below by $\exp(\epsilon r^{\alpha})$ for large $r$. 
Then, near zero,
\[
N(\lambda)\preceq\exp(-\lambda^{-\alpha/2}).	
\]
In particular, if $\Gamma$ has exponential growth, then 
$N(\lambda)\preceq\exp(-\lambda^{-1/2})$ near zero. 
\end{corollary}

This follows from 
Proposition~\ref{prop:folner decay implies doubling condition} 
and an inequality of Coulhon and Saloff-Coste~\cite{CouSal}, which implies
that, under the above hypothesis, 
$\folner(r)\succeq\exp(r^{\alpha})$ near infinity. 

\smallskip

Given any locally bounded function $f:(x_0,\infty)\rightarrow (0,\infty)$, we 
can find a continuous function 
$F:(x_0,\infty)\rightarrow (0,\infty)$ with $F(r)\ge f(r^2)$ for 
every $r\in (x_0,\infty)$. By making $F$ even bigger, we may also assume 
that $F(r)^2\le F(2r)$ and $F(r)\rightarrow\infty$ as $r\rightarrow\infty$. 
By a result of Erschler~\cite{erschler-piecewise}*{Theorem~1} 
there exists a finitely generated amenable group whose F{\o}lner's function 
exceeds $F$. Thus we obtain the following corollary:  

\begin{corollary}\label{cor:almost flat}
	For any positive, locally bounded function $f$ defined in a neighborhood 
	of $\infty$ 
	there exists a finitely generated amenable group $\Gamma$ such that,  
    near zero, 
	\[
		N(\lambda)\preceq\frac{1}{f(\lambda^{-1})}. 
	\]
\end{corollary}

This means that $N(\lambda)$ can be as flat as desired for amenable groups; 
on the other hand, $N(\lambda)$ is identically zero 
near zero if and only if the 
group is not amenable  
due to Kesten's spectral gap characterization. 

\subsection{Explicit computations}
\label{sub:explicit_computations}

The $N(\lambda)$-column of Table~\ref{table:computations} 
gives some samples of explicit computations 
obtained from 
Proposition~\ref{prop:folner decay implies doubling condition}, 
Proposition~\ref{pro:lower bound on N}, and 
Theorem~\ref{thm:isospectral}. Previous to the present work, 
estimates for $N(\lambda)$ were only known for virtually nilpotent 
groups~\cite{GroShu} (first row in Table~\ref{table:computations}) 
and for rank $1$ lamplighter groups 
with special 
generating sets~\cite{woess-bartoldi} 
(third row in Table~\ref{table:computations} with $d=1$). 
Many examples from Table~\ref{table:computations} are \emph{wreath products}.
We recall the definition.
If $\Gamma$ and $Q$ are groups, let $F(Q,\Gamma)$ be the set of functions from $Q$ 
to $\Gamma$ which are almost everywhere equal to the identity $e\in\Gamma$.
The semi-direct product $F(Q,\Gamma)\rtimes Q$ with respect to the action 
$(q\cdot f)(x)=f(q^{-1}x)$ is called 
the \emph{wreath product} of $\Gamma$ with $Q$ and is denoted by $\Gamma\wr Q$.
It is not difficult to check that the wreath product of two finitely generated groups
is finitely generated. 

\begin{table}\label{table:computations}
\begin{turn}{90}
\setlength\extrarowheight{5pt}
\begin{threeparttable}
\caption{\textbf{Computations for several classes of solvable groups}}
\begin{tabular}{>{\raggedright\hspace{0pt}}p{3.7cm}|p{4.2cm}|p{5cm}|p{3cm}|p{3cm}}
Group    &  $p(2t)$ as $t\rightarrow\infty$ & $N(\lambda)$ as $\lambda\rightarrow 0$ & $\Lambda(v)$ as $v\rightarrow\infty$ & $F(r)$ as $r\rightarrow\infty$\\
\hline
Virtually nilpotent of polynomial growth $d$ & $t^{-d/2}~~$\tnote{\bf(a)} & $\lambda^{d/2}$~~\tnote{\bf(b)} & $v^{-2/d}$~~\tnote{\bf(e)} & $r^d$~~\tnote{\bf(f)}\\
\hline
Virtually torsion-free solvable of exponential growth and finite Pr\"ufer rank\tnote{\bf(c)} & $\exp(-t^{1/3})$~~\tnote{\bf(d)} & $\exp(-\lambda^{-1/2})$ & $\log(v)^{-2}$~~\tnote{\bf(g)} & $\exp(r)$~~\tnote{\bf(h)}\\
\hline
$F\wr N$, $F$ finite and $N$ of polynomial growth $d$ & $\exp\bigl(-t^{\frac{d}{d+2}}\bigr)$~~\tnote{\bf(i)} & $\exp\bigl(-\lambda^{-d/2}\bigr)$ & $\log(v)^{-2/d}$~~\tnote{\bf(j)} & $\exp(r^d)$~~\tnote{\bf(j)}\\
\hline
$\Lambda\wr N$, $\Lambda$ infinite and of polynomial growth, $N$ of polynomial growth $d$ & 
$\exp\bigl(-t^{\frac{d}{d+2}}(\log(t))^{\frac{2}{d+2}}\bigr)$~~\tnote{\bf(k)} &
$\exp\bigl(-\lambda^{-\frac{d}{2}}\log(\frac{1}{\lambda})\bigr)$~~\tnote{\bf(n)} & 
$\bigl(\frac{\log(v)}{\log\log(v)}\bigr)^{-2/d}$~~\tnote{\bf(l)} &
$\exp\bigl(r^d\log(r)\bigr)$~~\tnote{\bf(l)}\\
\hline
$F\wr(\dots(F\wr(F\wr \mathbb Z))\dots)$, $F$ finite, $k$ times iterated wreath product, $k\ge 2$ & 
$\exp\left(-\frac{t}{(\log_{(k-1)}(t))^2}\right)$~~\tnote{\bf(m)} & $\exp\bigl(-\exp_{(k-1)}(\lambda^{-1/2})\bigr)$ & 
$\log_{(k)}(v)^{-2}$~~\tnote{\bf(l)} & 
$\exp_{(k)}(r)$~~\tnote{\bf(l)}\\
\hline
$\mathbb Z\wr(\dots(\mathbb Z\wr(\mathbb Z\wr \mathbb Z))\dots)$, $k$ times iterated wreath product, $k\geq 2$ & 
$\exp\left(-t\bigl(\frac{\log_{(k)}(t)}{\log_{(k-1)}(t)}\bigr)^2\right)$~~\tnote{\bf(m)} & $\exp\bigl(-\exp_{(k-1)}(\lambda^{-1/2}\log(\frac{1}{\lambda}))\bigr)$\tnote{\bf (n)} & 
$\bigl(\frac{\log_{(k)}(v)}{\log_{(k+1)}(v)}\bigr)^{-2}$~~\tnote{\bf(l)} & $\exp_{(k)}\bigl(r\log(r)\bigr)$~~\tnote{\bf(l)}
\end{tabular}
\medskip\small
\begin{tablenotes}[para]
\item[\bf(a)] This is a well known result of Varopoulos~\cite{varopoulos-book}. 
See~\cite{pittet-coulhon} for a short proof. 
\item[\bf(b)] This follows from~\citelist{\cite{GroShu}\cite{varopoulos-book}}. See also~\cite{lueck(2002a)}*{pp.~94-95}. 
\item[\bf(c)] Every polycyclic group of exponential growth is in this class. 
\item[\bf(d)] See~\cite{pittet-jems}. 
\item[\bf(e)] See~\cites{CouGri,pittet-coulhon}. 
\item[\bf(f)] See~\cites{revista, kaimanovich-vershik}.
\item[\bf(g)] See~\cites{CouGri,pittet-jems}. 
\item[\bf(h)] See~\cites{revista,pittet-jems}. 
\item[\bf(i)] See~\cite{pittet-wreath}. See also~\cite{pittet-coulhon} for an easier proof. 
\item[\bf(j)] See~\cite{pittet-isop} for the upper bound and~\cite{erschler} for the lower bound. 
\item[\bf(k)] See~\cite{pittet-wreath} for the lower bound and~\cite{erschler} for the upper bound. 
\item[\bf(l)] Proved in~\cite{erschler}. See~\cite{gromov-preprint} for an alternative proof. The notation $\exp_{k}$ stands for the $k$-times iterated 
exponential function. Similarly for $\log_{(k)}$. 
\item[\bf(m)] See~\cite{erschler}. 
\item[\bf(n)] This expression is different from but equivalent to the 
functional inverse of the reciprocal of the one right next to it. 
\end{tablenotes}
\end{threeparttable}
\end{turn}
\end{table}

\subsection{Stability}
The following theorem (proved in Section~\ref{sec:stability_results_for_spectral_distributions}) 
shows that the asymptotic behavior of $N_\mu$ near zero is an invariant of the group
(in the sense of \cite{GroAsy}). 

\begin{theorem}\label{thm:second moment}
Let $\Gamma$ be a finitely generated group. 
Let $\mu_1$  and $\mu_2$ be admissible 
probability measures on 
$\Gamma$. Then we have the dilatational equivalence near zero 
\[
    N_{\mu_1}\simeq N_{\mu_2}.
\]
\end{theorem}

Building on this, we prove in a forthcoming paper~\cite{companion} 
the invariance under quasi-isometry. 

\begin{theorem}\label{thm:QI stability}
Let $\Gamma$ and $\Lambda$ be finitely generated amenable groups. 
Let $\mu$ and $\nu$ be admissible probability measures on 
$\Gamma$ and $\Lambda$, respectively. 
If $\Gamma$ and $\Lambda$ are quasi-isometric, then, near zero,
\[N_\mu(\lambda)\simeq N_\nu(\lambda).\]
\end{theorem}

This result is an instance of a more general invariance 
result for arbitrary groups (i.e.~not necessarily amenable) 
with respect to uniform measure equivalence that also 
holds in any degree, not only for the Laplace operator on functions. 
In~\cite{companion} we also discuss how the stability of the 
return probability due to Pittet and 
Saloff-Coste~\cite{PitSalStab} (actually, a slightly 
stronger version thereof) 
can be deduced from Theorem~\ref{thm:QI stability}. 

\subsection{Structure of the paper}
All results of the paper are stated in full detail in the introduction.
The rest of the paper is devoted to their proofs.
Proposition \ref{pro:dual}, needed in the proof of Theorem \ref{thm:isospectral general},
as well as Proposition \ref{pro:lower bound on N}, which brings matching lower bounds
on $N(\lambda)$ for all examples from Table~\ref{table:computations}, 
may also be of independent interest.

\section{Properties of the Legendre transform} % (fold)
\label{sec:Abel}

We collect some elementary properties of the \emph{Legendre 
transform} for later reference. 

\begin{definition}[see~\cite{rockafellar}*{Section~26}]\label{def:Legendre}\hfill
\begin{enumerate}[(1)]
\item	Let $M:(0,\infty)\rightarrow (0,\infty)$ 
be decreasing such that $\lim_{x\rightarrow 0}M(x)=\infty$.
For $t>0$, we define the Legendre transform 
$\mathfrak{Le}_M(t)$ of $M$ as 
\[\mathfrak{Le}_M(t)=\inf\{tx+M(x);x>0\}.\]
\item
Let $G:[0,\infty)\rightarrow [0,\infty)$ 
be increasing 
such that $\lim_{x\rightarrow\infty}G(x)/x=0$.
For $t>0$, we define the Legendre conjugate transform 
$\mathfrak{Le}^*_G(t)$ of $G$ as 
\[\mathfrak{Le}^*_G(t)=\sup\{-tx+G(x);x\ge 0\}.\]
\end{enumerate}
\end{definition}

\begin{lemma}\label{lem:right}	
Let $M$ and $G$ be as in Definition~\ref{def:Legendre}. Let $t>0$. 
\begin{enumerate}[(1)]
\item If $M$ is right-continuous, then 
the infimum $\inf\{tx+M(x);x>0\}$ is a minimum. 
\item If $G$ is right-continuous, then 
the supremum $\sup\{-tx+G(x);x\ge 0\}$ is a maximum.
\end{enumerate}
\end{lemma}

\begin{proof}
The proofs of both assertions are very similar, and we only prove the 
first one as a sample. Since $f(x)=tx+M(x)\rightarrow\infty$ for both 
$x\rightarrow 0$ and $x\rightarrow\infty$, there exist $0<a\le b<\infty$ 
and a sequence $(x_n)_{n\in\bbN}$ in $[a,b]$ such that 
$f(x_n)\rightarrow\inf\{f(x);x>0\}$. By compactness of $[a,b]$, we 
may assume that $x_n\rightarrow y\in [a,b]$. We have to show that $y$ 
realizes the minimum. We may assume 
that either $x_n\ge y$ for every $n\in\bbN$ or $x_n\le y$ for every 
$n\in\bbN$. In the first case, $f(x_n)\rightarrow f(y)$ follows 
from right-continuity. In the second case, we have 
$tx_n+M(y)\le tx_n+M(x_n)$ since $M$ is decreasing. Taking 
limits, we obtain that $f(y)\le\inf\{f(x);x>0\}$. 
\end{proof}

We will need the following reformulation of~\cite{BCS}*{Lemma 3.2}.

\begin{lemma}\label{lem:BCS} 
Let $F$ be an increasing positive right-continuous function 
defined on $[0,\infty)$ which is bounded by $1$.
Assume moreover that $F(0)=0$ and that 
$F(\lambda)>0$ if $\lambda>0$. 
Let $M$ be the decreasing positive function 
defined on $(0,\infty)$ as $M(x)=-\log(F(x))$. 
Then, for all $t>0$, we have 
\begin{equation}\label{eq:Legendre}
\exp(-\mathfrak{Le}_M(t))\leq \int_0^{\infty}\exp(-t\lambda)dF(\lambda)
\leq (1+\mathfrak{Le}_M(t))\exp(-\mathfrak{Le}_M(t)).
\end{equation}
\end{lemma}

\begin{remark}
The hypothesis of right-continuity should be added 
in \cite{BCS}*{Lemma 3.2} and
the convexity assumption on the 
function $x\mapsto tx+M(x)$ should be removed.
(The convexity assumption is used only to ensure that the infimum
$\inf_{x>0}\{tx+M(x)\}$ is a minimum and the convexity assumption in the context of 
\cite{BCS}*{Lemma 3.2} implies anyway that 
$M$ is right-continuous.)
\end{remark}

\begin{proof}
To prove Lemma~\ref{lem:BCS}, apply Lemma~\ref{lem:right}; 
then follow the proof of~\cite{BCS}*{Lemma 3.2} without using
convexity (see the remark above). 
\end{proof}

\begin{proposition}\label{pro:bounds on M}
Let $M:(0,\infty)\rightarrow (0,\infty)$ 
be right-continuous and decreasing 
such that $\lim_{\lambda\rightarrow 0}M(\lambda)=\infty$. 
\begin{enumerate}[(1)]
\item		
Let $t_0>0$, and let $G:[t_0,\infty)\rightarrow (0,\infty)$ be 
a function such that, for every $t\ge t_0$, 
\[\mathfrak{Le}_M(t)\leq G(t).\]
If 
$G/\id:[t_0,\infty)\rightarrow (0,\infty)$, defined as $(G/\id)(t)=G(t)/t$,
has an inverse $(G/\id)^{-1}$, then 
\[M(\lambda)\leq G\circ(G/\id)^{-1}(\lambda)\]
holds for all $\lambda\in (G/\id)\bigl([t_0,\infty)\bigr)$.  
\item 
Let $G:[0,\infty)\rightarrow [0,\infty)$ 
be right-continuous and increasing 
such that $\lim_{t\rightarrow\infty}G(t)=\infty$ and  $\lim_{t\rightarrow\infty}G(t)/t=0$.
Assume that
\[\mathfrak{Le}_M(t)\geq G(t)\]
near infinity. Then, near zero, 
\[M(\lambda)\geq \mathfrak{Le}^*_G(\lambda).\]
\item
Let $G:[0,\infty)\rightarrow [0,\infty)$ be 
continuous such that $\lim_{t\rightarrow\infty}G(t)/t=0$. 
Assume that $t\mapsto G(t)/t$ is strictly decreasing near infinity.
Let $\epsilon\in (0,1)$. Then, near zero,  
\[\mathfrak{Le}^*_{G}(\lambda)\geq (1-\epsilon) G\circ(G/\id)^{-1}(\lambda/\epsilon).\]
\end{enumerate}
\end{proposition}

\begin{proof}
(1) By Lemma~\ref{lem:right}, for each $t\ge t_0$, there exists $y=y(t)>0$ such that $\mathfrak{Le}_M(t)=ty+M(y)$. Hence $y\leq G(t)/t$. As $M$ is decreasing, 
$M(G(t)/t)\le M(y)\le\mathfrak{Le}_M(t)\le G(t)$. \\
\noindent (2) Notice first that for every $t_0>0$ 
the function $t\mapsto -t\lambda+G(t)$ attains its maximum in $(t_0,\infty)$ 
provided $\lambda>0$ is sufficiently small. Indeed, let $T>t_0$ be such that 
$G(T)>1+G(t_0)$. Then for $\lambda\in (0,1/T)$ we have 
$-t\lambda+G(t)<-T\lambda+G(T)$ 
for every $t\in (0,t_0)$ because $G$ is increasing. Inequality (2) now readily follows 
from the definition of $\mathfrak{Le}_M$ and $\mathfrak{Le}^\ast_G$.\\
\noindent (3) By hypothesis, the inverse $(G/\id)^{-1}$ is well defined 
in a neighborhood of zero. For sufficiently small $\lambda>0$, there exists $s=s(\lambda)$, such that $\lambda s=\epsilon G(s)$, and thus 
$(G/\id)^{-1}(\lambda/\epsilon)=s$. 
Hence, 
$\mathfrak{Le}^*_G(\lambda)\ge -\lambda s+G(s)=(1-\epsilon)G(s)=(1-\epsilon)G\circ(G/\id)^{-1}(\lambda/\epsilon)$.
\end{proof}

\section{Proof of the main theorem}\label{sec:isospectral} % (fold)
Section~\ref{sec:isospectral} is devoted to the proof 
of Theorem~\ref{thm:isospectral general}, which 
implies Theorem~\ref{thm:isospectral}.

\subsection{The relation between $N_\mu(\lambda)$ and $p_\mu(t)$ 
via the Laplace transform} 
\label{subsec:Laplace transform}
								
We discuss the following equivalence near infinity 
between the return probability of the random walk associated to $\mu$ and 
the Laplace transform of $N_\mu$. 

\begin{proposition}\label{pro:dual}
Let $\Gamma$ be a finitely generated infinite amenable 
group and let $\mu$ be an admissible 
probability measure on $\Gamma$. 
Then, for $t\in\bbN$ near infinity, 
\begin{equation}\label{eq:return probabilities} 
p_\mu(2t)\simeq\int_{0}^{\infty}\exp(-\lambda t)dN_\mu(\lambda).
\end{equation}
\end{proposition}

This equivalence is well known -- at least, if the Markov operator 
$R_\mu$ is positive (but it is often not, like 
in the case of the simple 
random walk $\mu=(\delta_1+\delta_{-1})/2$ on $\Gamma=\bbZ$). In practice, 
it is often sufficient to apply the equivalence to the positive operator 
$R_\mu^2$, but we do need the general case here. 
The proof relies on our stability 
result (Theorem~\ref{thm:second moment}). Firstly we need the 
following two lemmas. 

\begin{lemma}\label{lem:two tails}
Let $\calh$ be a separable Hilbert space and let
$A\in B(\calh)$ be a self-adjoint operator such that
$\|A\|\leq 1$. For each $0\leq\lambda\leq 1$, the spectral projections
of $I-A^2$ and of $I-A$ satisfy:
$$E_{\lambda}^{I-A^2}=
E_{1-\sqrt{1-\lambda}}^{I-A}+I-E_{1+\sqrt{1-\lambda}}^{I-A}.$$
\end{lemma}

\begin{proof}
By definition of the spectral projections, it is enough to
check the corresponding equality at the level of functions.
But if $-1\leq x\leq 1$ and if $0\leq\lambda\leq 1$, then 
\[\chi_{(-\infty,\lambda]}(1-x^2)=
\chi_{(-\infty,1-\sqrt{1-\lambda}]}(1-x)+
1-\chi_{(-\infty,1+\sqrt{1-\lambda}]}(1-x).\qedhere\]
\end{proof}

\begin{lemma}\label{lem:partial integration}
Let $F, G:[a,b]\rightarrow\bbR$ be two functions of bounded variation. 
Assume that $F(a)=G(a)$ and that $F(x)\leq G(x)$ on $[a,b]$. 
If $f:[a,b]\rightarrow [0,\infty)$ is a
continuous decreasing function, 
then the Stieltjes integrals of
$f$ with respect to $F$ and $G$ satisfy:
\[\int_{[a,b]}f(x)dF(x)\leq \int_{[a,b]}f(x)dG(x).\]
\end{lemma}

\begin{proof}
Integration by parts reads (see \cite{riesz}*{Chapitre II. 54})
$$\int_{[a,b]}f(x)dF(x)=f(x)F(x)|^b_a-\int_{[a,b]}F(x)df(x).$$
The hypotheses $F(a)=G(a)$ and $\forall x\in[a,b],\, F(x)\leq G(x)$,
and the assumption that $f$ is non-negative imply that 
\[f(x)F(x)|^b_a\leq f(x)G(x)|^b_a.\]
On the other hand, as $f$ is decreasing and 
as $\forall x\in[a,b],\, F(x)\leq G(x)$, we obtain:
\[-\int_{[a,b]}F(x)df(x)\leq -\int_{[a,b]}G(x)df(x).\qedhere\]
\end{proof}

\begin{proof}[Proof of the $\succeq$-assertion of Proposition~\ref{pro:dual}]
	Let $0<\lambda_0<1$ be sufficiently small such that 
	\[
		-\frac{\log(1-\lambda)}{\lambda}=1+\frac{\lambda}{2}+\frac{\lambda^2}{3}+\cdots <\frac{3}{2}~\text{ for every $\lambda\in (0,\lambda_0)$.} 
	\]
   This implies that 
	\[
		\exp(-3t\lambda)<(1-\lambda)^{2t}~\text{ for every $t\ge 1$ and $0\leq\lambda<\lambda_0$.}
	\]
	It is important that the choice of $\lambda_0$ does not depend on $t$. 
	As $\Gamma$ is infinite, $N_{\mu}(0)=0$. Notice also that as $\Gamma$ is amenable,
	if $\lambda>0$ then $N(\lambda)>0$. In particular $N(\lambda_0)>0$.  We obtain for every $t\in\bbN$ that 
	\begin{align*}
		\int_{[0,\infty)}\exp(-3t\lambda)dN_\mu(\lambda)&=
		\int_{[0,\lambda_0]}\exp(-3t\lambda)dN_\mu(\lambda)+
		\int_{(\lambda_0,2]}\exp(-3t\lambda)dN_\mu(\lambda)\\
		&\leq\int_{[0,\lambda_0]}\exp(-3t\lambda)dN_\mu(\lambda)+
		\exp(-3t\lambda_0)(N_\mu(2)-N_\mu(\lambda_0))\\
		&\leq\int_{[0,\lambda_0]}\exp(-3t\lambda)dN_\mu(\lambda)+
	    \frac{N_\mu(2)-N_\mu(\lambda_0)}{N_\mu(\lambda_0)}\int_{[0,\lambda_0]}\exp(-3t\lambda)dN_\mu(\lambda)\\	
	&=\bigl(1+\frac{N_\mu(2)-N_\mu(\lambda_0)}{N_\mu(\lambda_0)}\bigr)\int_{[0,\lambda_0]}\exp(-3t\lambda)dN_\mu(\lambda)\\	
					&<\bigl(1+\frac{N_\mu(2)-N_\mu(\lambda_0)}{N_\mu(\lambda_0)}\bigr)\int_{[0,\infty)}(1-\lambda)^{2t}dN_\mu(\lambda)\\
					&=	\bigl(1+\frac{N_\mu(2)-N_\mu(\lambda_0)}{N_\mu(\lambda_0)}\bigr)p_\mu(2t).\qedhere
	\end{align*} 
\end{proof}
\begin{proof}[Proof of the $\preceq$-assertion of Proposition~\ref{pro:dual}]
	The support of $\mu^{(2)}$ (convolution of $\mu$ with itself) 
	either generates $\Gamma$ 
	or a subgroup in $\Gamma$ of index $2$. (To prove it, notice that the support of the admissible measure $\mu$ contains a finite symmetric generating set $S$ of $\Gamma$.
	Hence $S^2\subset \supp(\mu)^2\subset \supp(\mu^{(2)})$. But the subgroup $H$ generated by $S^2$ is of index at most $2$ in $\Gamma$ because, for any $s\in S$, $H\cup Hs=\Gamma$.)
	
	 In the first case,  
	Theorem~\ref{thm:second moment} immediately yields 
	that $N_{\mu^{(2)}}\simeq N_\mu$ in the dilatational sense. 
	In the latter case,   
	Theorem~\ref{thm:second moment} and~\cite{lueck(2002a)}*{Theorem~2.55 (6) on p.~98} 
	imply that $N_{\mu^{(2)}}\simeq 2N_\mu$ holds dilatationally. 
	Hence there are $1\geq\lambda_1>0$ and $D\ge 1$ such that 
	\begin{equation}\label{eq:self convolution}
	N_{\mu^{(2)}}(\lambda)\le 2N_\mu(D\lambda)~\text{ for every $\lambda\in [0,\lambda_1]$.}
	\end{equation}
One verifies that 
\begin{equation}\label{eq:exp estimate}
(1-\lambda)^{2t}\le\exp(-2t\lambda)~\text{ for every $t\ge 1$ and every 
$\lambda\in [0,1]$.}
\end{equation}
Set $\lambda_0=\min\{1,\lambda_1/2\}$. 
Also notice that $(1-\lambda)^{2t}\le (1-\lambda_0)^{2t}$ 
for $\lambda_0\le\lambda\le 2-\lambda_0$. 
Thus,  
\begin{align}\label{eq:splitting the integral} 
	p_\mu(2t)&\overset{\eqref{eq:measure and density}}{=}\int_{[0,\lambda_0]}(1-\lambda)^{2t}dN_\mu(\lambda)+\int_{(\lambda_0,2-\lambda_0]}(1-\lambda)^{2t}dN_\mu(\lambda)+\underbrace{\int_{(2-\lambda_0,2]}(1-\lambda)^{2t}dN_\mu(\lambda)}_{=R(\lambda_0)}\notag\\
	&\leq\int_{[0,\lambda_0]}(1-\lambda)^{2t}dN_\mu(\lambda)+\int_{(\lambda_0,2-\lambda_0]}(1-\lambda_0)^{2t}dN_\mu(\lambda)+R(\lambda_0)\notag\\
	&\le \int_{[0,\lambda_0]}\exp(-2t\lambda)dN_\mu(\lambda)+ \exp(-2t\lambda_0)+ R(\lambda_0)\\
	&\le \bigl(1+N_\mu(\lambda_0)^{-1}\bigr)\int_{[0,\lambda_0]}\exp(-2t\lambda)dN_\mu(\lambda)+R(\lambda_0).\notag
\end{align}

(Again, since $\Gamma$ is amenable, there is no spectral gap, hence 
$N_\mu(\lambda_0)>0$.) 
The estimate of $R(\lambda_0)$ relies on two observations: 

i) Applying Lemma~\ref{lem:two tails} to the Markov-operator $R_\mu=1-\Delta_\mu$, 
we obtain for every $0\leq\lambda\leq\lambda_1$ that 
\begin{align}\label{eq:estimate at 2}
	1-N_\mu\bigl(2-\frac{\lambda}{2}\bigr)&\le 1-N_\mu\bigl(2-\frac{\lambda}{2}\bigr)+N_\mu\bigl(\frac{\lambda}{2}\bigr)\\
		&\le 1-N_\mu\bigl(2-\frac{\lambda}{1+\sqrt{1-\lambda}}\bigr)+N_\mu\bigl(\frac{\lambda}{1+\sqrt{1-\lambda}}\bigr)\notag\\
		&=1-N_\mu(1+\sqrt{1-\lambda})+N_\mu(1-\sqrt{1-\lambda})\notag\\
		&=N_{\mu^{(2)}}(\lambda)\notag\\
		&\overset{\eqref{eq:self convolution}}{\le} 2N_{\mu}(D\lambda).\notag  
\end{align} 

ii) We conclude from the change of variable $\lambda\mapsto 2-\lambda$ and Lemma~\ref{lem:partial integration} 
and~\eqref{eq:estimate at 2} that 

\begin{align*}
	R(\lambda_0)&=\int_{[0,\lambda_0)}(1-\lambda)^{2t}d\bigl(1-N_\mu(2-\lambda)\bigr)\\
	&\le 2\int_{[0,2D\lambda_0)}\bigl(1-\frac{\lambda}{2D}\bigr)^{2t}dN_\mu(\lambda)\\
	&\overset{\eqref{eq:exp estimate}}{\le} 2\int_{[0,2D\lambda_0)}\exp\bigl(-t\frac{\lambda}{D}\bigr)dN_\mu(\lambda). 
\end{align*}
Combining the latter estimate with~\eqref{eq:splitting the integral} and using 
that $D\ge 1$ we obtain that 
\[
	p_\mu(2t)\le \bigl(3+N_\mu(\lambda_0)^{-1}\bigr)\int_{[0,2D\lambda_0)}\exp\bigl(-t\frac{\lambda}{D}\bigr)dN_\mu(\lambda)~\text{ for every $t\in\bbN$.}\qedhere
\]
\end{proof}

\subsection{General properties of doubling functions and generalized inverses} % (fold)
\label{sub:smoothing_and_regularity_assumptions}

\begin{definition}\label{def:generalized inverse}
Let $L:(0,\infty)\rightarrow (0,\infty)$ be a decreasing function. 
\begin{enumerate}[(1)]
\item The \emph{generalized inverse} of $L$ is the function 
$L^{-1}:(\inf L,\infty)\rightarrow (0,\infty)$ 
defined as $L^{-1}(x)=\inf\bigl\{v>0;~L(v)\le x\bigr\}$.
\item We say that $L$ is \emph{doubling (near infinity)} if there exists a constant $c>0$ such 
that $L(2x)\ge cL(x)$ for all $x>0$ (sufficiently large $x>0$). 
\end{enumerate}
\end{definition}

\begin{lemma}\label{lem:smoothing}
Let $L:(0,\infty)\rightarrow (0,\infty)$ be a decreasing function. Assume that $L$ is doubling. Then there is a 
continuous, decreasing 
function $L_{ct}: (0,\infty)\rightarrow (0,\infty)$ 
such that $L_{ct}\simeq L$. If, in addition, $L\circ\exp$ is 
doubling, then $L_{ct}\circ\exp$ is doubling. 
\end{lemma}

\begin{proof} 
The function 
\[
	L_{ct}(x)=\frac{2}{x}\int_{x/2}^xL(s)ds
\] 
is continuous and 
differentiable almost everywhere by the Lebesgue differentiation theorem. 
It is decreasing because of 
\begin{align*} \frac{xL_{ct}^{'}(x)}{2}&=L(x)-\frac{L(x/2)}{2}-\frac{1}{x}\int_{x/2}^xL(s)ds\\
&\le L(x)-\frac{L(x/2)}{2}-\frac{L(x)}{2}=
\frac{L(x)-L(x/2)}{2}\le 0.
\end{align*} 
It is equivalent to $L$ 
because of $L_{ct}(x)\ge L(x)$ and $L_{ct}(x)\le L(x/2)\le c^{-1}L(x)$.
The last assertion also follows from this.
\end{proof}

\begin{lemma}\label{lem:generalities on inverses}
Let $f,g:(0,\infty)\rightarrow (0,\infty)$ be decreasing functions 
such that $f(x)\rightarrow 0$ and $g(x)\rightarrow 0$ for $x\rightarrow\infty$. 
So the generalized inverses $f^{-1}$ and $g^{-1}$ are defined on $(0,\infty)$. 
Then the following holds: 
\begin{enumerate}[(1)]
\item If $f\circ exp$ is doubling near infinity, then $f$ is doubling near infinity.
\item Let $l:(0,\infty)\rightarrow (0,\infty)$ 
be a strictly increasing, continuous function (i.e.~$l^{-1}$ in the usual sense exists) such that $l(x)\rightarrow\infty$ as $x\rightarrow\infty$. 
Then $(f\circ l)^{-1}=l^{-1}\circ f^{-1}$. 
\item If $f\simeq g$ and $f,g$ are doubling near infinity, then there is 
a constant $D>0$ such that $D^{-1}g(x)\le f(x)\le Dg(x)$ for large $x>0$. 
\item If $f\preceq g$ near infinity, then $f^{-1}\preceq g^{-1}$ near zero. 
\item If there is $D>0$ with $f(x)\le Dg(x)$ near infinity, then $f^{-1}(\lambda)\le g^{-1}(D^{-1}\lambda)$ near zero. 
\item If $f\simeq g$ and $f$ and $g$ are doubling near infinity, then 
$f^{-1}\simeq g^{-1}$ holds near zero in the dilatational sense.   
\end{enumerate}
\end{lemma}

\begin{proof}
\noindent (1) If $x\geq 2$, then  
\[f(2x)=f\bigl(\exp(\log(2x))\bigr)\ge f\bigl(\exp(2\log(x))\bigr)
		\ge cf\bigl(\exp(\log(x)\bigr)=cf(x).\]
\noindent (2) By continuity, there is $x_0>0$ with 
$l(0,\infty)=(x_0,\infty)$. Hence, if $\lambda>0$ is small enough, we 
have 
\[
	\inf\bigl\{y>0;~f(y)\le\lambda\bigr\}=\inf\bigl\{l(x);~f(l(x))\le\lambda, x>0\bigr\}.
\]
Thus, as $l^{-1}$ is increasing and continuous, 
\begin{equation*} 
	\bigl(l^{-1}\circ f^{-1}\bigr)(\lambda) = l^{-1}\bigl(\inf\{l(x);~f(l(x))\le\lambda, x>0\}\bigr)
	=\inf\{x>0;~f(l(x))\le\lambda\}=(f\circ l)^{-1}(\lambda).
\end{equation*} 
\noindent (3)--(5) are easy; the proof is omitted. \\
\noindent (6) follows from (3) and~(5). 
\end{proof}

\subsection{Coulhon's and Grigoryan's functional equation}\label{sub:functional equation} 
In the proof of Theorem \ref{thm:isospectral general} we relate 
the $L^2$-isoperimetric profile to the return probability. 
This relies on work of Coulhon and Grigor'yan~\cite{CouGri}.

\begin{proposition}\label{prop:functional equation} 
Let $L:(0,\infty)\rightarrow (0,\infty)$ be decreasing and continuous. 
Assume that $\lim_{x\rightarrow\infty}L(x)=0$. 
Then the functional equation 
\begin{equation}\label{eq:functional equation}
	t=\int_0^{v(t)}\frac{ds}{(L\circ\exp)(s)}
\end{equation}
uniquely defines 
a strictly increasing $C^1$-function 
$v: [0,\infty)\rightarrow [0,\infty)$. Further, 
the following properties hold: 
\begin{enumerate}[(1)]
	\item $v(2t)\le 2v(t)$ for $t\ge 0$. 
	\item If $L\circ\exp$ is doubling near infinity, then 
	there is $c>0$ with $v'(2t)\ge cv'(t)$ near infinity. 
	\item For $t>0$, the function $t\mapsto v(t)t^{-1}$ is decreasing and 
	strictly decreasing near infinity. 
\end{enumerate}
\end{proposition} 

\begin{proof} 
\noindent (1) Since the derivative
$v'(t)=(L\circ\exp)(v(t))$ is decreasing, 
\begin{align*}
	v(2t)=\int_0^{2t}v'(s)ds +v(0)\leq \int_0^{2t}v'(s/2)ds +v(0)
	&\leq 2\int_0^tv'(s)ds +2v(0)\\&=2v(s). 	
\end{align*} 
\noindent (2) We have
	\[v'(2t)=(L\circ\exp)(v(2t))\geq (L\circ\exp)(2v(t))
	\geq c (L\circ\exp)(v(t))=cv'(t).\]
\noindent (3) Differentiating yields 
	\begin{align*} \frac{d}{dt}\frac{v(t)}{t}
		=\frac{v(t)}{t^2}\bigl(\frac{v'(t)t}{v(t)}-1\bigr)
		&=\frac{v(t)}{t^2}\bigl(\frac{(L\circ\exp)(v(t))t}{v(t)}-1\bigr)\\
&\le\frac{v(t)}{t^2}\Bigl(\frac{t}{\int_0^{v(t)}\frac{ds}{(L\circ\exp)(s)}}-1\Bigr)=\frac{v(t)}{t^2}\bigl(\frac{t}{t}-1\bigr)=0.
	\end{align*}
Let us show that the inequality 
\[\int_0^{v(t)}\frac{ds}{(L\circ\exp)(s)}\leq\frac{v(t)}{(L\circ\exp)(v(t))}\]
we just used is strict provided $t$ is large enough. Assume the equality case 
for some $t>0$. Then, for every $x\in [0,t]$,  
we have $(L\circ\exp)(v(x))=(L\circ\exp)(v(t))$ since $L\circ\exp\circ v$ 
is decreasing. This is a contradiction for large $t>0$ 
because $L(x)\rightarrow 0$ as $x\rightarrow\infty$. 
\end{proof}

\begin{theorem}[Coulhon-Grigor'yan]\label{thm:CouGri} Let $\Gamma$ be 
	an infinite, finitely generated, amenable group, and let $\mu$ be an admissible measure on $\Gamma$. 
	Let $L:(0,\infty)\rightarrow (0,\infty)$ be a decreasing, continuous function, and let $v(t)$ be defined by~\eqref{eq:functional equation}. 
	\begin{enumerate}[(1)]
		\item If $L(v)\preceq\Lambda_{\mu}(v)$ near infinity, then 
   $p_\mu(t)\preceq \exp(-v(t))$ near infinity. 
		\item Assume, in addition, that $L\circ\exp$ is doubling.
		If $L(v)\succeq\Lambda_{\mu}(v)$ near infinity, then 
        $p_\mu(2t)\succeq\exp(-v(t))$. 
	\end{enumerate}
\end{theorem}

\begin{proof}
This is a straightforward application 
of~\cite{CouGri}*{Section~3} (see also~\cite{pittet-coulhon}*{Propositions~2.1 and~4.1}); the technical condition (D) of~\cite{CouGri}*{Section~3}, needed for the lower bounded on $p_\mu(t)$, is satisfied because of 
Proposition~\ref{prop:functional equation} (1) and~(2). 
\end{proof}

\begin{lemma}\label{lem:v(t)/t Lemma}
Retain the setting of Proposition~\ref{prop:functional equation}. 
Assume $L\circ\exp$ is doubling near infinity. Then there is a constant 
$D>0$ such that near infinity 
\[
	(L\circ\exp\circ v)(t)\le\frac{v(t)}{t}\le D\cdot(L\circ\exp\circ v)(t). 
\]
In particular, we have  
\[
	\lim_{t\rightarrow\infty}\frac{v(t)}{t}=0.
\]
\end{lemma}

\begin{proof}
Since $L$ is decreasing, we obtain that 
\begin{equation*}	
	\frac{v(t)}{t}=v(t)\Bigl (\int_0^{v(t)}\frac{ds}{(L\circ\exp)(s)}\Bigr)^{-1}\ge  
	(L\circ\exp)(v(t)).
\end{equation*}
If $C>0$ is the doubling constant, then we obtain for every $v>0$ that 
\[\int_0^v\frac{ds}{(L\circ\exp)(s)}\ge\int_{v/2}^v\frac{ds}{(L\circ\exp)(s)}\ge  
\frac{v}{2(L\circ\exp)(v/2)}\ge\frac{v}{2C(L\circ\exp)(v)}.\]
Thus, 
\begin{align*}
\frac{v(t)}{t}=v(t)\Bigl(\int_0^{v(t)}\frac{ds}{(L\circ\exp)(s)}\Bigr)^{-1}
&\le v(t)\Bigl(\frac{v(t)}{2C(L\circ\exp\circ v)(t)}\Bigr)^{-1}\\
&=2C(L\circ\exp\circ v)(t).\qedhere
\end{align*}
\end{proof}

\subsection{Conclusion of the proof of Theorem~\ref{thm:isospectral general}} % (fold)
\label{sub:conclusion_of_the_proof_of_the_main_formula}

By assumption, $L\circ\exp$, thus, by 
Lemma~\ref{lem:generalities on inverses} (1), $L$ itself, are doubling near infinity. 
Let $L_1=L_{ct}$ be the continuous, decreasing function obtained 
from Lemma~\ref{lem:smoothing}; $L_1\circ\exp$ is also doubling near infinity 
and $L_1\simeq L$. 
Let $v_1$ be the continuous, increasing function defined by the 
functional equation~\eqref{eq:functional equation} where we replace $L$ 
with $L_1$. The proof now splits into two parts. 

\begin{proof}[Proof of the upper bound on $N_\mu(\lambda)$ in Theorem \ref{thm:isospectral general}] 

Due to Theorem~\ref{thm:CouGri}, we have 
$p_\mu(t)\preceq\exp(-v_1(t))$ near infinity since $L_1\simeq L\preceq\Lambda_\mu$. 
Combining this with Proposition~\ref{pro:dual}, 
we obtain that 
\[-\log\Bigl(\int_0^{\infty}\exp(-\lambda t)dN_\mu(\lambda)\Bigr)\succeq v_1(t)\]
near infinity. 
Since $\Gamma$ is infinite and amenable, $N_\mu(0)=0$ and $N_\mu(\lambda)>0$ for every $\lambda>0$. 
Hence, Lemma~\ref{lem:BCS} applies, and we deduce that near infinity,
\[\mathfrak{Le}_M(t)\succeq v_1(t),\]
where, by definition, $M(x)=-\log(N_\mu(x))$.
In other words, there exists $\alpha, \beta>0$, such that near infinity, 
\[\mathfrak{Le}_M(t)\geq v_2(t),\]
where $v_2(t)=\alpha v_1(\beta t)$. 
Define the function $L_2$ by 
\begin{equation}\label{eq:variable substitution}
L_2\circ\exp(x)=\alpha\beta (L_1\circ\exp)(\alpha^{-1} x). 
\end{equation}
One verifies that 
$v_2$ and $L_2$ satisfy the assumptions 
in Proposition~\ref{prop:functional equation} and, in particular, 
the functional equation~\eqref{eq:functional equation}. Furthermore, 
$L_2\circ\exp$ is doubling near infinity because $L_1\circ\exp$ is so. 
By Lemma~\ref{lem:v(t)/t Lemma}, $v_2(t)/t\rightarrow 0$. 
Hence Proposition~\ref{pro:bounds on M} (2) implies that \[M(\lambda)\geq\mathfrak{Le}^*_{v_2}(\lambda)\] 
near zero. 
By Proposition~\ref{prop:functional equation} and 
Lemma~\ref{lem:v(t)/t Lemma}, the function
$v_2(t)/t$ is strictly decreasing near infinity and converges 
to $0$. Thus its inverse $(v_2/\id)^{-1}$ is well defined near zero. 
Proposition~\ref{pro:bounds on M} (3) applies, and we 
deduce that
\[\mathfrak{Le}^*_{v_2}(\lambda)\ge \frac{1}{2}v_2\circ (v_2/\id)^{-1}\bigl(2\lambda\bigr)\]
near zero. By Lemma~\ref{lem:v(t)/t Lemma}, 
$v_2/\id\ge L_2\circ\exp\circ v_2$. Applying, in this order,  
Lemma~\ref{lem:generalities on inverses}~(5) and~(2), yields 
\[(v_2/\id)^{-1}\ge (L_2\circ\exp\circ\, v_2)^{-1}=
v_2^{-1}\circ(L_2\circ\exp)^{-1}\]
near zero. 
Let $L_3:(0,\infty)\rightarrow (0,\infty)$ be the function defined by 
$L_3(\exp(x))=L_2(\exp(2x))$. 
Putting all together and using 
Lemma~\ref{lem:generalities on inverses}~(2) for the last two equalities below, 
we obtain that 
\begin{align*}
-\log\bigl(N_\mu(\lambda/2)\bigr)=M(\lambda/2)\ge\mathfrak{Le}^\ast_{v_2}(\lambda/2)
&\ge\frac{1}{2}(L_2\circ\exp)^{-1}(\lambda)\\
&=(L_3\circ\exp)^{-1}(\lambda)=\log\circ L_3^{-1}(\lambda)
\end{align*} 
near zero. Thus, near zero, 
\[N_\mu(\lambda)\preceq\frac{1}{L_3^{-1}(\lambda)}\]
holds in the dilatational sense. It remains to see that $L_3^{-1}\simeq 
L^{-1}$ near zero in the dilatational sense. 
Since near infinity $L_3\simeq L_2\simeq L_1\simeq L$ 
and all these functions are doubling, this follows from 
Lemma~\ref{lem:generalities on inverses}~(6). 
\end{proof}

\begin{proof}[Proof of the lower bound on $N_\mu(\lambda)$ in 
	Theorem \ref{thm:isospectral general}] 
Since $L_1\simeq L\succeq\Lambda_\mu$ and $L_1\circ\exp$ is doubling at 
infinity, Theorem~\ref{thm:CouGri} can be applied and yields that 
$p_\mu(2t)\succeq\exp(-v_1(t))$ near infinity. We proceed similarly as for the 
upper bound and apply Proposition~\ref{pro:dual} and 
Lemma~\ref{lem:BCS}. This implies 
\[
v_1(t)\succeq -\log\Bigl(\int_0^\infty\exp(-t\lambda)dN_\mu(\lambda)\Bigr)
\ge\mathfrak{Le}_M(t)-\log\bigl(1+\mathfrak{Le}_M(t)\bigr)\ge\frac{1}{2}\mathfrak{Le}_M(t)
\]
for large $t>0$. Thus there are constants $\alpha,\beta>0$ such that 
for $v_2(t)=\alpha v_1(\beta t)$ we have 
\[v_2(t)\ge\mathfrak{Le}_M(t)\]
for large $t>0$. Let $L_2$ be defined as in~\eqref{eq:variable substitution}. 
With the same argument as for the upper bound, 
the function $v_2/\id$ has a well defined inverse near zero. 
Hence Proposition~\ref{pro:bounds on M}~(1) can be applied, 
implying that near zero:
\[
	M(\lambda)\le v_2\circ\bigl(v_2/\id\bigr)^{-1}(\lambda).
\]
By Lemmas~\ref{lem:v(t)/t Lemma} and~\ref{lem:generalities on inverses}~(5) and~(2), 
\[
	(v_2/\id)^{-1}(\lambda)\le v_2^{-1}\circ(L_2\circ\exp)^{-1}(\lambda/D).
\]
We conclude that 
\[
	M(D\lambda)\le\log\circ L_2^{-1}(\lambda).
\]
near zero. Now we proceed exactly as for the upper bound, but with reversed inequalities. 
\end{proof}

\section{F{\o}lner's functions and Cheeger's inequality}\label{sec:isoperimetric profile and geometry}
In this section, we combine Theorem~\ref{thm:isospectral general} with some geometric tools and prove Proposition~\ref{prop:folner decay implies doubling condition}. 
This leads to estimates of the spectral distribution 
in terms of the F{\o}lner's function~\eqref{eq:Folner function}, the growth function, and F{\o}lner couples. 

Throughout this section, $\Gamma$ denotes a finitely generated group 
and $S$ denotes a finite symmetric generating set of $\Gamma$. 
The Laplace operator $\Delta$, 
the spectral distribution $N$, and the $L^2$-isoperimetric profile 
$\Lambda$ are taken with respect to the probability 
measure~\eqref{eq:simple random walk}. For 
statements up to equivalence, the specific choice of 
an admissible probability measure does not matter 
(Theorem~\ref{thm:second moment}). 

We refer to~\cite{dodziuk}*{Theorem~2.3} 
for a proof of the combinatorial
version of \emph{Cheeger's inequality}: 
\begin{theorem}[Cheeger's inequality]\label{thm:Cheeger inequality}
\[\lambda_1(\Omega)\geq\frac{1}{2|S|^2}
\left(\inf_{\omega\subset\Omega}\frac{\abs{\partial_S\omega}}{\abs{\omega}}\right)^2.\]
\end{theorem}

\begin{lemma}\label{lem:regularity of F}
Let $F$ be a continuous strictly increasing positive function
defined on a neighborhood of infinity with $\lim_{r\to\infty}F(r)=\infty$.
Let $\alpha>0$ and $\beta\ge 0$. For large $v>0$ we define 
\[L(v)=\frac{\alpha}{(F^{-1}(v)-\beta)^2}.\]
Assume that there exists a constant $C>1$ such that for large $r>0$ 
\[
	F(Cr)\geq F(r)^2.
\] 
Then the function $L\circ\exp$ is doubling near infinity.	
\end{lemma}

\begin{proof} In a neighborhood of zero, we have:
\[
	L^{-1}(\lambda)=F(\alpha^{1/2}\lambda^{-1/2}+\beta).
\]
This function is dilatationally equivalent to $F(\lambda^{-1/2})$. 
Hence we have the dilatational equivalence
\[
	\log\circ\, L^{-1}(\lambda)\simeq\log\circ\, F(\lambda^{-1/2}).
\]
But,
\[
	2\log(F(\lambda^{-1/2}))=\log(F(\lambda^{-1/2})^2)\leq
	\log(F(C\lambda^{-1/2}))=\log(F((C^{-2}\lambda)^{-1/2})).
\]
Hence there is a constant $\delta>0$ such that, near zero,
\[
	2\log\circ\, L^{-1}(\lambda)\leq\log\circ L^{-1}(\alpha\lambda).
\]
From this we deduce that $L\circ \exp$ is doubling near infinity. 	
\end{proof}

\begin{proof}[Proof of Proposition~\ref{prop:folner decay implies doubling condition}] 
We will apply Theorem~\ref{thm:isospectral general}~(1) with
\[
L(v)=\frac{1}{F^{-1}(v)^2}.	
\]
The doubling of $L\circ\exp$ follows
from Lemma~\ref{lem:regularity of F}. Hence the corollary will be proved if we show that 
\[
\Lambda(v)\succeq L(v),	
\]
for large $v>0$.
Let $\alpha\geq 1, \beta\geq 1$ such that 
$F(r)\leq\alpha\folner(\beta r)$ near infinity. 
As $\folner$ increases and 
$\lim_{r\to\infty}\folner(r)=\infty$, for a given large $v>0$, 
there exists $x>0$ such that 
\[
\folner(x-1)<v\leq\folner(x).	
\]
Hence, if $\abs{\Omega}<v$, then 
\[
\frac{|\partial_S\Omega|}{|\Omega|}\geq\frac{1}{x}.	
\]
On the other hand, as $F^{-1}$ is strictly increasing, we obtain
\[
	(x-1)/\beta=F^{-1}\bigl(F((x-1)/\beta)\bigr)\le F^{-1}\bigl(\alpha \folner(x-1)\bigr)<F^{-1}(\alpha v).
\] 
Cheeger's inequality implies that
\[
	\Lambda(v-1)\geq \frac{1}{2|S|^2}\min_{|\Omega|\leq v-1}\left(\frac{|\partial_S\Omega|}{|\Omega|}\right)^2.
\]
We deduce that
\[
	\Lambda(v/2)\geq \frac{1}{2|S|^2(\beta F^{-1}(\alpha v)+1)^2}, 
\]	
which yields $\Lambda\succeq L$. 
\end{proof}

\begin{proposition}\label{pro:lower bound on N} 
Let $\Gamma$ be a finitely generated group. 
Assume there exists a sequence of finite subsets
$(\Omega_n)_{n\in\mathbb N}$ of $\Gamma$ with the following properties: 
\begin{enumerate}
	\item There exists a constant $\alpha\geq 1$, such that $\lambda_1(\Omega_n)\leq \alpha/n^2$.
	\item $\abs{\Omega_n}<\abs{\Omega_{n+1}}$.
	\item There exists $C>1$ such that the piecewise linear extension $F:[1,\infty)\rightarrow\bbR$ 
	of the function $n\mapsto\abs{\Omega_n}$ 
	satisfies $F(Cr)\geq F(r)^2$ for large $r>0$.
\end{enumerate}
Then, near zero,
\[
N(\lambda)\succeq\frac{1}{F(\lambda^{-1/2})}.	
\]	
\end{proposition}

\begin{proof} Define  
\[
L(v)=\frac{\alpha}{(F^{-1}(v)-1)^2}.		
\]
The inverse function 
$L^{-1}(\lambda)=F(\alpha^{1/2}\lambda^{-1/2}+1)$ 
is dilatationally equivalent to $F(\lambda^{-1/2})$ near zero.
By Lemma~\ref{lem:regularity of F} the 
function $L\circ\exp$ is doubling near infinity. 
Hence the assertion will follow from 
Theorem~\ref{thm:isospectral general}~(2) once it has been shown that $\Lambda(v)\le L(v)$ near infinity. 
Let $v>1$. Let $n\in\bbN$ be determined by 
$F^{-1}(v)-1<n\leq F^{-1}(v)$. 
From the definition of $F$ we get 
\[
	\abs{\Omega_{n}}=F(n)\le F(F^{-1}(v))=v.
\]
But, by hypothesis, 
\[
	\lambda_1(\Omega_{n})\leq\frac{\alpha}{n^2}\leq\frac{\alpha}{(1-F^{-1}(v))^2}=L(v).\qedhere
\]
\end{proof}

The above proposition is inspired from~\cite{erschler}*{Proposition~2} 
and~\cite{pittet-coulhon}*{Theorem~4.7}. As mentioned in~\cite{erschler}*{Proposition~2}, the required upper bound $\lambda_1(\Omega_n)\leq\frac{\alpha}{n^2}$ holds in the case there exists
$\omega_n\subset\Omega_n$ such that $d_S(\omega_n,\Gamma\setminus\Omega_n)>\epsilon n$, and such that $|\Omega_n|\leq C|\omega_n|$, where $\epsilon>0, C\geq 1$ are constants independent of $n$.
The pairs $(\Omega_n,\omega_n)$ are called \emph{F{\o}lner 
couples}. In several examples, this approach leads to lower bounds on 
$N(\lambda)$ that match the upper bounds deduced from 
Proposition~\ref{prop:folner decay implies doubling condition}.  
In particular, for all the examples listed in Table~\ref{table:computations} it 
leads to matching bounds. We refer the reader 
to~\cites{erschler,pittet-coulhon,PitGDF,pittet-jems} for the construction of 
F{\o}lner couples.

\section{Stability} % (fold)
\label{sec:stability_results_for_spectral_distributions}

This section is devoted to the proof of Theorem~\ref{thm:second moment}. 
It is possible to deduce this theorem from~\cite{GroShu}*{Proposition~4.1} 
if both measures are assumed to 
have finite support. The general case requires a direct proof: 

\begin{proof}[Proof of Theorem~\ref{thm:second moment}] 		
For a probability measure $\mu$ on $\Gamma$, consider the following 
bounded, $\Gamma$-equivariant operator
\begin{equation*}%\label{eq:first-differential}
 C_\mu:\bigoplus_{s\in\supp(\mu)}l^2(\Gamma)\xrightarrow{\oplus_{s\in
    \supp(\mu)}(\mu(s)/2)^{1/2}R_{s-1}}l^2(\Gamma),
\end{equation*}
where $R_{s-1}:l^2(\Gamma)\rightarrow l^2(\Gamma)$ is right convolution 
by $\delta_s-\delta_e$ (or in other words, right multiplication by $s-1$). 
If $\mu$ is symmetric, we compute 
\begin{align}\label{eq:Laplacian via differential}
	C_\mu C_\mu^\ast =\sum_{s\in\supp(\mu)}\frac{\mu(s)}{2}R_{2-s-s^{-1}}
		&= \sum_s\mu(s)\id-\frac{1}{2}\sum_s\mu(s)R_s-\frac{1}{2}\sum_s\mu(s)R_{s^{-1}}\notag\\
		&= 1-\sum_s\mu(s)R_s\\
		&= \Delta_\mu\notag. 
\end{align}
So we have to show that the spectral distributions of 
$C_{\mu_1}C_{\mu_1}^\ast$ and $C_{\mu_2}C_{\mu_2}^\ast$ 
are equivalent near zero. 
Without loss of generality, we may and will assume that 
the density of $\mu_2$ is $\sum_{s\in S}\delta_s/\abs{S}\in l^1(\Gamma)$ 
for a finite, symmetric generating set $S$. 
Throughout, direct sums of Hilbert spaces are understood to be completed direct sums. 

Firstly, we show that there is a bounded, $\Gamma$-equivariant operator 
$F$ that makes the following square commutative. 
\[
\xymatrix{
\bigoplus_{h\in\supp(\mu_1)}l^2(\Gamma)\ar[d]^F\ar[r]^-{C_{\mu_1}}     
& l^2(\Gamma)\ar[d]^\id\\ 
\bigoplus_{s\in S}l^2(\Gamma)\ar[r]^-{C_{\mu_2}} 
& l^2(\Gamma)
}
\]
To this end, 
choose for every $\gamma\in\Gamma$ a path $w^\gamma$ in the Cayley graph 
$\operatorname{Cayl}(\Gamma,S)$ 
from $e$ to $\gamma$ of length $n=l(\gamma)$. Further, 
let $(v_0^\gamma,\ldots, v_n^\gamma)$ denote the successive 
vertices of $w_\gamma$, and $(w_1^\gamma,\ldots, w_n^\gamma)$ 
the successive oriented edges of $w^\gamma$. Note that 
$v_0^\gamma=e$ and $w_i^\gamma$ has endpoints $v_{i-1}^\gamma$ and 
$v_i^\gamma$. 
The matrix representation 
$F=(F_{s,h})_{h\in\supp(\mu_1), s\in S}$ of $F$ is given by 
\[
    F_{s,h}=(\abs{S}\mu_1(h))^{1/2}\sum_{1\le i\le l(h), w_i^h=s}R_{v^h_{i-1}}. 
\]
The map $F$ is well defined on the dense subset 
$\bigoplus_{s\in\supp(\mu_1)}\bbC\Gamma$. 
It is straightforward to verify that 
$F$ is bounded with operator norm  
\[
    \norm{F}^2\le\sum_{h\in\supp(\mu_1)}\abs{S}^2\mu_1(h)l(h)^2<\infty.  
\]
To see that this $F$ makes the above square commutative, by $\bbC\Gamma$-linearity, 
one only has to verify that 
for the unit element $1_h\in \bigoplus_{\supp(\mu_1)}l^2\Gamma$ in the copy of 
$l^2\Gamma$ associated to some $h\in\supp(\mu_1)$ we have 
\begin{equation*}%\label{eq:commutative statement}
	C_{\mu_2}\bigl(\sum_{s\in S} F_{s,h}(1_h)\bigr)=C_{\mu_1}(1_h)=\bigl(\frac{\mu_1(h)}{2}\bigr)^{1/2}(h-1). 
\end{equation*}
But this follows from: 
\begin{align*}
	C_{\mu_2}\bigl(\sum_{s\in S} F_{s,h}(1_h)\bigr) &=
	\sum_{s\in S}\bigl(\frac{\mu_2(s)\abs{S}\mu_1(h)}{2}\bigr)^{1/2}\sum_{1\le i\le l(h), w_i^h=s}v^h_{i-1}(s-1)\\
	&=\bigl(\frac{\mu_1(h)}{2}\bigr)^{1/2}\sum_{s\in S}\sum_{1\le i\le l(h), w_i^h=s}v^h_{i-1}(s-1)\\
	&=\bigl(\frac{\mu_1(h)}{2}\bigr)^{1/2}\sum_{1\le i\le l(h)}v^h_{i}-v^h_{i-1}\\
	&=\bigl(\frac{\mu_1(h)}{2}\bigr)^{1/2}(h-1).
\end{align*}

Let $\alpha:\bbC\Gamma\rightarrow\bbC$ be the linear map uniquely defined 
by $\alpha(\gamma)=1$ for every $\gamma\in\Gamma$. 
For a finite generating set $T\subset\Gamma$, consider the map 
\[
	\phi_T:\bigoplus_{t\in T}\bbC\Gamma\xrightarrow{\oplus R_{t-1}}\bbC\Gamma. 
\]
It satisfies $\Im(\phi_T)=\ker(\alpha)$. 
This is either proved directly (see~\cite{brown}*{Exercise~1 on p.~12}) or by noting that 
$\phi_T$ is the first differential in a free $\bbC\Gamma$-resolution 
of $\bbC$ (or, topologically, 
in the cellular chain complex of a model of the universal 
space $E\Gamma$ that has $\operatorname{Cayl}(\Gamma,T)$ as 
its $1$-skeleton~\cite{brown}*{(4.3)~Example on p.~16}). 
By assumption there is a finite generating set $T\subset\supp(\mu_1)$. So 
we have 
\[
	\Im\bigl(C_{\mu_1}\vert_{\bigoplus_{t\in T}\bbC\Gamma}\bigr)=\Im(\phi_T)=\ker(\alpha). 
\]
Similarly, we have 
\[
	\Im\bigl(C_{\mu_2}\vert_{\bigoplus_{s\in S}\bbC\Gamma}\bigr)=\ker(\alpha).
\]
For every $s\in S$, pick $x_s\in\bigoplus_{t\in T}\bbC\Gamma$ with 
$C_{\mu_1}(x_s)=C_{\mu_2}(1_s)\in\ker(\alpha)$ where $1_s\in \Gamma\subset\bbC\Gamma$ is the unit element in the $s$-th component of 
$\bigoplus_{s\in S}\bbC\Gamma$. Define $G:\bigoplus_{s\in S}\bbC\Gamma\rightarrow \bigoplus_{t\in T}\bbC\Gamma$ to be the unique 
$\Gamma$-equivariant, linear map with $G(1_s)=x_s$ for every $s\in S$. 
Then $G$ is 
bounded. We obtain a commutative square: 
\[
	\xymatrix{
	\bigoplus _{h\in\supp(\mu_1)}l^2(\Gamma)\ar[r]^-{C_{\mu_1}}     
	& l^2(\Gamma)\\ 
	\bigoplus_{s\in S}l^2(\Gamma)\ar[r]^-{C_{\mu_2}}\ar[u]^G
	& l^2(\Gamma)\ar[u]^\id.
	}
\]
Let $\pr_1:\bigoplus _{h\in\supp(\mu_1)}l^2(\Gamma)\circlearrowleft$ be the 
projection onto the orthogonal complement $\ker(C_{\mu_1})^\perp$ of 
$\ker(C_{\mu_1})$, and similarly, let $\pr_2$ be the projection 
onto $\ker(C_{\mu_2})^\perp$. 
Then we obtain two commutative squares  
\begin{equation*}\label{eq: pair of squares}
\xymatrix{
   \ker(C_{\mu_1})^\perp\ar[r]^-{C_{\mu_1}}\ar[d]^{\pr_2\circ F}\ar[r] &
   l^2(\Gamma)\ar[d]^\id &  & \ker(C_{\mu_1})^\perp\ar[r]^-{C_{\mu_1}}\ar[r] &
   l^2(\Gamma)\\
   \ker(C_{\mu_2})^\perp\ar[r]^-{C_{\mu_2}}\ar[r] &
   l^2(\Gamma) & & \ker(C_{\mu_2})^\perp\ar[r]^-{C_{\mu_2}}\ar[u]^{\pr_1\circ G} &
   l^2(\Gamma)\ar[u]^\id.
}
\end{equation*} 
The commutativity and the 
injectivity of $C_{\mu_1}$ and $C_{\mu_2}$ when restricted to $\ker(C_{\mu_1})^\perp$ and $\ker(C_{\mu_1})^\perp$, respectively, 
already 
imply that $\pr_2\circ F$ is an isomorphism with inverse $\pr_1\circ G$. 
One easily verifies that 
\begin{equation}\label{eq:adjoint without kernel}
\bigl(C_{\mu_i}\vert_{\ker^\perp}\bigr)\circ \bigl(C_{\mu_i}\vert_{\ker^\perp}\bigr)^\ast=C_{\mu_i}\circ C_{\mu_i}^\ast. 
\end{equation}

We apply 
Efremov-Shubin's 
\emph{min-max principle}~\citelist{\cite{GroShu}*{(1.3)}~\cite{lueck(2002a)}*{Definition~2.1 and Lemma~2.3 on pp.~73-74}} which 
says that 
\[
	\trace_\Gamma\bigl(E^{A^\ast A}_{\lambda^2}\bigr)=
	\sup\bigl\{\trace_\Gamma(\pr_L);~\text{$L\subset l^2(\Gamma)$ closed 
	$\Gamma$-invariant, $\norm{Ax}\le\lambda\norm{x}~\forall x\in L$}\bigr\}.
\]
to the operator 
$A=\bigl(C_{\mu_1}\vert_{\ker(C_{\mu_1})^\perp}\bigr)^\ast=
(\pr_2\circ F)^\ast \bigl(C_{\mu_2}\vert_{\ker(C_{\mu_2})^\perp}\bigr)^\ast$. 
From this, combined 
with~\eqref{eq:Laplacian via differential} 
and~\eqref{eq:adjoint without kernel}, we obtain, setting 
$B=\bigl(C_{\mu_2}\vert_{\ker(C_{\mu_2})^\perp}\bigr)^\ast$ 
and $T=(\pr_2\circ F)^\ast$, that 
\begin{align*}
	N_{\mu_1}\bigl(\lambda^2\bigr) &=
	\sup\bigl\{\trace_\Gamma(\pr_L);~L\subset l^2(\Gamma), \fixnorm{A x}=\fixnorm{TBx}\le\lambda\norm{x}~\forall x\in L\bigr\}\\
	&\le\sup\bigl\{\trace_\Gamma(\pr_L);~L\subset l^2(\Gamma), \norm{Bx}=\fixnorm{T^{-1}TBx}\le\fixnorm{T^{-1}}\lambda\norm{x}~\forall x\in L\bigr\}\\
	&=N_{\mu_2}\bigl(\fixnorm{T^{-1}}^2\lambda^2\bigr).
\end{align*}
A similar argument yields 
	$N_{\mu_2}\bigl(\lambda^2\bigr)\le N_{\mu_1}\bigl(\fixnorm{T}^2\lambda^2\bigr)$. 
\end{proof}

\end{document}